\documentclass[10pt]{amsart}

\usepackage{amsfonts,amssymb,amscd,amstext}
\usepackage[colorlinks=true,linkcolor=blue,citecolor=blue]{hyperref}
\usepackage[utf8]{inputenc}
\usepackage{mathrsfs}
\usepackage{dsfont,xcolor}

\renewcommand{\le}{\leqslant}
\renewcommand{\ge}{\geqslant}
\newcommand{\ptl}{\partial}
\newcommand{\rr}{{\mathbb{R}}}
\newcommand{\la}{\lambda}
\newcommand{\La}{\Lambda}
\newcommand{\vfi}{\varphi}
\newcommand{\h}{H}

\newcommand{\Sg}{\Sigma}

\newcommand{\Om}{\Omega}

\newcommand{\ga}{\gamma}

\newcommand{\de}{\delta}

\newcommand{\volf}[1]{\textup{vol}_\Psi (#1)}
\newcommand{\areaf}[1]{\text{area}_\Psi (#1)}
\newcommand{\perf}[1]{P_\Psi (#1)}

\newcommand{\volpsi}{\volume_\Psi}

\renewcommand{\geq}{\geqslant}
\renewcommand{\leq}{\leqslant}

\DeclareMathOperator{\divv}{div}
\DeclareMathOperator{\volume}{vol}

\DeclareMathOperator{\Jac}{Jac}

\DeclareMathOperator{\Isom}{Isom}
\DeclareMathOperator{\conv}{conv}
\DeclareMathOperator{\symm}{sym}

\newtheorem{theorem}{Theorem}[section]

\newtheorem{lemma}[theorem]{Lemma}

\theoremstyle{definition}

\newtheorem{remark}[theorem]{Remark}

\theoremstyle{remark}

\numberwithin{equation}{section}

\begin{document}

\title{Isoperimetric inequalities in cylinders with density}

\author[K.~Castro]{Katherine Castro} \address{Departamento de
Geometr\'{\i}a y Topolog\'{\i}a \\
Universidad de Granada \\ E--18071 Granada \\ Espa\~na}
\email{ktcastro@ugr.es}


\date{\today}

\thanks{The author has been supported by MICINN-FEDER grants MTM2010-21206-C02-01, MTM2013-48371-C2-1-P, MEC-Feder grants MTM2017-84851-C2-1-P and PID2020-118180GB-I00, and Junta de Andalucía grants A-FQM-441-UGR18 and P20-00164. Open Access funding: Universidad de Granada / CBUA}

\begin{abstract}
Given a compact Riemannian manifold with density $M$ without boundary and the real line $\rr$ with constant density, we prove that isoperimetric regions of large volume in $M\times\rr$ with the product density are slabs of the form $M\times [a,b]$. We previously prove, as a necessary step, the existence of isoperimetric regions in any manifold of density where a subgroup of the group of transformations preserving weighted perimeter and volume acts cocompactly.
\end{abstract}

\subjclass[2000]{}

\maketitle

\thispagestyle{empty}

\bibliographystyle{abbrv} 


\section{Introduction and preliminaries}
In  recent  years,  isoperimetric  problems have been considered in manifolds  with  density. One of the most  interesting spaces of this type is the Gauss space, the Euclidean space $\rr^n$ with the Gaussian density $\Psi(x):=\exp(-\pi|x|^2)$. Borell  \cite{MR399402} and Sudakov and Tirel’son  \cite{ST} independently proved in 1974 and 1975 that half-spaces minimize perimeter under a volume constraint for this density. A new proof was given in 1983 by Ehrhard \cite{MR745081} using symmetrization. In 1997 Bobkov \cite{MR1428506} proved a functional version of this isoperimetric inequality, later extended to the sphere and used to prove isoperimetric estimates for the unit cube by Barthe and Maurey \cite{MR1785389}. Following \cite{MR1428506}, Bobkov and Houdré \cite{MR1396954} considered “unimodal densities”with finite total measure on the real line. These authors explicitly computed the isoperimetric profile for such densities and found some of the isoperimetric solutions. Gromov \cite{MR1978494,MR2481743} studied manifolds with density as “metric measure spaces” and mentioned the natural generalization of mean curvature obtained by the first variation of weighted area. Bakry and Ledoux \cite{MR1374200} and Bayle \cite{bayle-thesis} proved generalizations of the Lévy\textendash Gromov isoperimetric inequality and other geometric comparisons depending on a lower bound on the generalized Ricci curvature of the manifold. Isoperimetric comparisons results in manifolds with density were considered by Maurmann and Morgan \cite{MR2507612}. Existence of isoperimetric sets in $\rr^n$ with density under various hypotheses on the growth of the density were proven by  Morgan and Pratelli \cite{MR3038539} and Milman \cite{milman}, see also De Philippis, Franzina and Pratelli \cite{guidofranzpratelli}. For regularity of isoperimetric regions with density see Sect. 3.10 in paper of Morgan \cite{MR1997594} and see also Pratelli and Saracco \cite{PRATELLI2018}. Boundedness of isoperimetric regions was studied by Cinti and Pratelli  \cite{CintiPratelli2017} and Pratelli and Saracco \cite{PratelliSaracco}. Symmetrization techniques in manifolds with density developed by Ros \cite{rosisoperimetric} and Morgan et al. \cite{MR2895338}.

For nice surveys on manifolds with density the reader is referred to \cite{MR2161354,bayle-thesis,MR3735647} and the references therein.

\medskip 

In this paper, $(M,g,\Psi)$ will denote a manifold with density without boundary, where $g$ is a Riemannian metric on $M$ and $\Psi:M\to\rr$ is a smooth function. We define the weighted \emph{volume} of a set by
\begin{equation}
\label{eq:volumef}
\volf{E}:=\int_E e^\Psi dM,
\end{equation}
where $dM$ is the Riemannian volume element on $(M,g)$. The weighted \emph{area} of a smooth hypersurface $\Sg$ is defined by
\begin{equation}
\label{eq:areaf}
\areaf{\Sg}:=\int_\Sg e^\Psi d\Sg,
\end{equation}
where $d\Sg$ is the Riemannian area element on $\Sg$.

If $E\subset M$, we define the weighted \emph{perimeter} of $E$ in the manifold with density $(M,g,\Psi)$ by
\begin{equation}
\label{eq:perpsi}
\perf{E}:=\sup\bigg\{\int_E\divv_\Psi X\,dM: X\in\frak{X}_0^\infty(M),\ |X|\le 1\bigg\},
\end{equation}
where $\frak X^\infty_0\left(M \right)$ is the set of smooth vector fields on $M$ with compact support and
\begin{equation}
\label{eq:divpsi}
\divv_\Psi(X)=\divv(e^\Psi X),
\end{equation}
and $\divv$ is the Riemannian divergence in $(M,g)$. If $E$ has smooth boundary $\Sg$, then $\perf{E}=\areaf{\Sg}$, see \cite{MR3038539}.

Given a manifold with density, we shall denote by $\Isom(M,g,\Psi)$ the group of isometries of $(M,g)$ preserving the function $\Psi$ (i.e., maps $f:M\to M$ such that $\Psi\circ f=\Psi$). Such isometries preserve the weighted area and volume.

The isoperimetric profile of $(M,g,\Psi)$ is the function $I:[0,+\infty) \to \rr^+$ defined by
\begin{equation}
\label{eq:profile}
I(v)=\inf\{\perf{E}:\volf{E}=v\}
\end{equation}

A set $E\subset M$ of finite weighted perimeter is isoperimetric if $\perf{E}=I(\volf{E})$. This means that $E$ minimizes the weighted perimeter under a weighted volume constraint. Regularity of isoperimetric sets was considered by Morgan and Pratelli \cite{MR3038539}.

Given a manifold with density $(M,g,\Psi)$, we shall consider the cylinders with density $(M\times\rr^k,g\times g_0,\Psi\times 1)$, where $\rr^k$ is $k$-dimensional Euclidean space with its standard Riemannian metric $g_0$, and $(\Psi\times 1)(p,x)=\Psi(p)$ for every $(p,x)\in M\times\rr^k$. Given $v\in\rr^k$, we define $t_v:M\times\rr^k\to M\times \rr^k$ by $t_v(p,x):=(p,x+v)$ for any $(p,x)\in M\times\rr^k$. The set $G:=\{t_v:v\in\rr^k\}$ is contained in $\Isom(M\times\rr^k,g\times g_0,\Psi\times 1)$. In case $M$ is a compact manifold, the quotient of $M\times \rr^k$ by $\Isom(M\times\rr^k,g\times g_0,\Psi\times 1)$ is the compact base $M$ of the product. We focus in this paper in the case $k=1$. 

The aim of this paper is to prove that isoperimetric sets in $(M\times\rr,g\times g_0,\Psi\times 1)$ are slabs of the form $M\times [a,b]$, where $a,b\in\rr$, $a<b$. This result is proven in Theorem~\ref{thm:cylinder} in Section~\ref{sec:slabs}. As a necessary previous step in our proof we must show existence of isoperimetric regions in manifolds with density such that the action of $\Isom(M,g,\Psi)$ is cocompact, that is, the quotient $(M,g,\Psi)/\Isom(M,g,\Psi)$ is compact, like in the case of the cylinders considered in Section~\ref{sec:slabs}. The proof of existence is based on Galli and Ritoré's  in contact sub-Riemannian manifolds, see \cite{gr}.  Since this proof has now become standard, we check in Section~\ref{sec:existence} that the main ingredients are available: a relative isoperimetric inequality for uniform radii, see Theorem~\ref{thm:relisop}; the doubling property, see Theorem~\ref{thm:dp}; and a deformation result for sets of finite perimeter, see Theorem~\ref{DL}. In Theorem~\ref{thm:cylinder} we  characterize the isoperimetric regions of \emph{large volume} in a cylinder with density $M\times\rr$, where $M$ is a compact Riemannian manifold with density, the real line $\rr$ is endowed with a constant density, and the product with the product density. In the non-weighted Riemannian case Duzaar and Steffen \cite{DS}  proved that in $M\times\rr$ isoperimetric sets of large volume are of the form $M\times[a,b]$, where $a,b\in\rr$. For higher dimensional Euclidean factors the problem was considered in \cite{RVLARGE}, where the authors proved that the isoperimetric solutions of large volume in the Riemannian product $M\times\rr^k$ are of the form $M\times B(x,r)$, where $B(x,r)$ is an Euclidean ball, see also \cite{gonzalo2013}.

For more results about variational problems in cylinders the reader is referred to  \cite{fmp} and the references in \cite{ROSALES2021}.

%
%
%
%



%

\section{Existence of isoperimetric regions in $M$}
\label{sec:existence}

In this Section we prove the existence of isoperimetric sets, for any volume, in a manifold with density $(M,g,\Psi)$ such that $\Isom(M,g,\Psi)$ acts cocompactly. The scheme of proof devised by Galli and Ritoré in \cite{gr} applies to our situation, provided we are able to show that
\begin{itemize}
\item There exists $r_0>0$ such that a relative isoperimetric inequality holds in all balls $B(p,r)$, with $p\in M$, $0<r\le r_0$, with a uniform constant.
\item The manifold is doubling. This means the existence of $r_0>0$ and a uniform constant $C_D>0$ such that $\volf{B(p,2r)}\le C_D \volf{B(p,r)}$ for all $p\in M$ and $0<r\le r_0$.
\item A deformation result for finite perimeter sets, see Theorem~\ref{DL}, holds in $(M,g,\Psi)$.
\end{itemize}
Assuming these results hold in $(M,g,\Psi)$, and using the well-known techniques in \cite{gr} we have the following result

\begin{theorem}
\label{thm:existence}
In a manifold with density $(M,g,\Psi)$ such that $\Isom(M,g,\Psi)$ acts cocompactly, isoperimetric sets exist for any given volume.
\end{theorem}

To prove the required ingredients needed for Theorem~\ref{thm:existence} we start with a preliminary result. We recall that the \emph{convexity radius} $\conv(K)$ of a subset $K$ of a Riemannian manifold M is the infimum of positive numbers $r$ such that the geodesic open ball $B(p,r)$ is convex for every $p\in K$. We call $d$  the Riemannian distance in $(M,g)$.

\begin{lemma}
\label{laLa}
Let $(M,g)$ be a Riemannian manifold, and $K\subset M$ a compact subset. Let $r_0=\conv(K)$. Then there exist 
functions $\la$, $\La:[0,r_0]\to\rr$ such that $1+\la$, $1+\La$ are positive, $\lim_{r\to 0}\la(r)=\lim_{r\to 0}\Lambda(r)=0$, and

\begin{equation}
\label{eq:lL}
(1+\la(r))\,|x-y|\le d(\exp_p(x),\exp_p(y))\le (1+\Lambda(r))\,|x-y|,
\end{equation}
for any $p\in K$ and $x,y\in B(0,r_0)\subset T_pM$.
\end{lemma}

\begin{proof}
Given $p\in K$, consider a compact coordinate neighborhood $U$ around $p$ with a global orthonormal basis defined on $U$. Given $q\in U$, let $g_{ij}^q$ be the components of the Riemann tensor in the coordinate neighborhood defined by the exponential map $\exp_q:B(0,r_0)\subset T_qM\to B(q,r_0)$ and the global orthornormal basis. The functions $g_{ij}^q$ depend smoothly on $q$. We define
\begin{equation}
\begin{split}
\alpha_U(r):=&\min_{q\in U}\bigg\{\bigg(\sum^n_{i,j=1} g^q_{ij}(z)\,v_iv_j\bigg)^\frac12: |z|\leq r < r_0, \sum^n_{i=1} v^2_i=1\bigg\},
\\
\beta_U(r):=&\max_{q\in U}\bigg\{\bigg(\sum^n_{i,j=1} g^q_{ij}(z)\,v_iv_j\bigg)^\frac12: |z|\leq r < r_0, \sum^n_{i=1} v^2_i=1 \bigg\}.
\end{split}
\end{equation}
It is easy to check that $\alpha_U(r)$ is decreasing, $\beta_U(r)$ is increasing, and that \[
\lim_{r\to 0} \alpha_U(r)=\lim_{r\to 0}\beta_U(r)=1.
\]

Given $q\in U$, we take $ x,y \in B(0,r_0)$. 
To compute the distance $d$ between the points $\exp_q(x)$, $\exp_q(y)$, it is enough to consider curves inside the convex ball $B(0,r)$. Let  $\gamma :I \rightarrow B(0,r_0) $ be a curve joining $x$ and $y$. Then
\begin{equation}
d(\exp _q (x),\exp _q (y))=L(\exp_q\circ\gamma)=\int_I \bigg( \sum^n_{i,j=1} g^q_{ij} (\gamma(t)) \, \gamma_i'(t)\gamma_j'(t) \bigg)^\frac12 dt.
\end{equation}
Observe that
\begin{equation}
\alpha_U(r)\bigg(\sum_{i=1}^n\ga_i'(t)^2\bigg)^{1/2}\le \sum^n_{i,j=1} g^q_{ij}(\gamma(t)) \, \gamma_i'(t)\gamma_j'(t)\le \beta_U(r)\bigg(\sum_{i=1}^n\ga_i'(t)^2\bigg)^{1/2}.
\end{equation}
The left quantity is larger than or equal to $\alpha_U(r)\,|x-y|$. Since $\gamma$ is an arbitrary curve joining $x$ and $y$, this implies $\alpha_U(r)\,|x-y|\le d(\exp_q(x),\exp_q(y))$. On the other hand,
\[
d(\exp_q(x),\exp_q(y))\le L(\exp_q\circ\gamma)\le \beta_U(r)\,|x-y|,
\]
so we have
\begin{equation}
\alpha_U (r) \leq \frac{d(\exp_q(x),\exp_q(y))}{|x-y|} \leq \beta_U (r).
\end{equation}
The result follows by covering the compact set $K$ by a finite number of coordinate neighborhoods $U$, and taking $\alpha(r)$ as the minimum of the $\alpha_U(r)$ and $\beta(r)$ as the maximum of the $\beta_U(r)$. Setting $\lambda(r)=\alpha(r)-1$, $\Lambda(r)=\beta(r)-1$ the result follows.
\end{proof}

Using Lemma~\ref{laLa} we are obtain as corollaries the existence of a uniform relative isoperimetric inequality and the existence of a doubling constant.

\begin{theorem}[Relative isoperimetric inequality in $(M,g,\Psi)$]
\label{thm:relisop}
Let $(M,g,\Psi)$ be an $n$-dimensional manifold with density and $K\subset M$  a compact subset. Let $r_0>0$ be the radius obtained in Lemma~\ref{laLa}. Then for all $p\in K$ 
and $E\subset B(p,r)$ with $0<r\leq r_0$, there exists a positive constant $C>0$ not depending on $p$, such that
\begin{equation}
\label{eq:R}
\perf {E,B(p,r)}\geq C\cdot \min {\big \{ \volf{E},\, \volf{B(p,r)\setminus E}\big \} }^{(n-1)/n}.
\end{equation}

In particular, \eqref{eq:R} holds in the whole manifold if $\Isom(M,g,\Psi)$ acts cocompactly on $M$, see Lemma 3.5 in \cite{gr}.
\end{theorem}

\begin{proof}
By Lemma~\ref{laLa}, for $0<r\le r_0$, the exponential map $f=\exp_p:B(0,r)\to B(p,r)$ is a diffeomorphism. Let $F=f^{-1}(E)$. Then 
\begin{equation}
\volf{f(F)}=\volf{E}=\int_E e^\Psi d\h^n, 
\end{equation}
where $H^n$ is the $n$-dimensional Hausdorff measure in $(M,g)$. Consider positive functions constants $a,b>0$ so that $a\leq e^\Psi \leq b$ in $\overline{B}(p,r_0)$ for all $p\in K$. So we have
\begin{equation}
a\cdot \h^n (E)=a \int_E d\h^n\leq \int_E e^\Psi d\h^n\leq b\int_E  d\h^n=b\cdot \h^n (E).
\end{equation}
As $f$ is Lipschitz,  Lemma  \ref{laLa} in dimension $n$ implies
\begin{equation}
\label{eq:H}
a(1+\la(r))^n\h^n_0(F)\leq a\cdot \h^n (E)\leq \volf{E}\leq b\cdot \h^n (E)\leq b(1+\La(r))^n\h^n_0(F),
\end{equation}
where $\h^n_0$ is the $n$-dimensional Hausdorff measure  with respect to the Euclidean metric.

On the other hand, by \S~2 in \cite{MR3038539}, 
$$
\perf{E,B(r,p)}=\int_{\ptl_\star E\cap B(p,r)} e^\Psi d\h^{n-1},
$$
where $\ptl_\star E$ is the reduced boundary of $E$.
And therefore,
\begin{align*}
\perf{E,B(p,r)} &\leq b\h^{n-1}(\ptl_{\star} E\cap B(p,r)) \\
                         &\leq b(1+\La(r))^{n-1} \h^{n-1}_0(\ptl_{\star} F\cap B(0,r)) \\
                         &= b(1+\La(r))^{n-1} P_0(F,B(0,r)),
\end{align*}
where $P_0$ is the Euclidean perimeter.
By a similar computation we obtain 
$$
\perf{E,B(p,r)}\geq a(1+\la(r))^{n-1}P_0(F,B(0,r)).
$$
Then
$$
a(1+\la(r))^{n-1}P_0(F,B(0,r))\leq \perf{E,B(p,r)} \leq b(1+\La(r))^{n-1} P_0(F,B(0,r)).
$$
Observe that
\begin{align*}
\perf{E,&B(p,r)}  \geq a(1+\la(r))^{n-1}P_0(F,B(0,r)) \\
                        & \geq a(1+\la(r))^{n-1}\cdot C_0 \cdot \min{\big \{ \h^n(F),\,\h^n(B(0,r)\setminus F) \big \}^{(n-1)/n}}\\
                        & \geq a(1+\la(r))^{n-1}\cdot C_0 \cdot \min{\bigg\{\frac {\volf{E}}{(b(1+\La(r)))^n}, \frac{\volf{B(p,r)\setminus E}}{(b(1+\La(r)))^n}\bigg\}^{\tfrac{n-1}{n}}},
\end{align*}
where $C_0$ is the constant in the relative isoperimetric (Poincaré) inequality in Euclidean balls. Thus,
$$
\perf{E,B(p,r)}\geq\frac{a(1+\la(r))^{n-1}}{b(1+\La(r))^{n-1}}\cdot C_0 \cdot \min{\big \{ \volf{E}, \volf{B(p,r)\setminus E}\big \}^{\tfrac{n-1}{n}}}.\qedhere
$$
\end{proof}

In the following result we prove that $(M,g,\Psi)$ is a doubling metric space.

\begin{theorem}[Doubling property] 
\label{thm:dp}
Let $(M,g,\Psi)$ be an $n$-dimensional manifold with density and $K\subset M$  a compact subset. Let $r_0>0$ be the radius obtained in Lemma~\ref{laLa}. Then there exists a constant $C_D>0$, only depending on $K$, such that for all $x_0\in K$ and $0<r\le r_0/2$ we have
\begin{equation} {\label{eq:DP}}
\volf{B\left(x_{0},2r\right)}\leq C_D\volf{B\left(x_{0},r\right)}.
\end{equation}

In particular, \eqref{eq:DP} holds in the whole manifold if $\Isom(M,g,\Psi)$ acts cocompactly on $M$.
\end{theorem}

\begin{proof} From equation \eqref{eq:H} we know that, for $r_{0}$, $0<r<r_{0}$, we 
must have
\[
AH_{0}^{n}\left(F\right)\leq\volf{E}\leq BH_{0}^{n}\left(F\right),
\]
where $A=a\left(1+\lambda\left(r\right)\right)^{n}$, $B=b\left(1+\Lambda\left(r\right)\right)^{n}$
and $\mbox{exp}_{x_{0}}\left(F\right)=E$ for all $x_{0}\in M$.

In particular,
\[
AH_{0}^{n}\left(B\left(0,2r\right)\right)\leq\volf{B\left(x_{0},2r\right)}\leq BH_{0}^{n}\left(B\left(0,2r\right)\right)
\]
and
\[
AH_{0}^{n}\left(B\left(0,r\right)\right)\leq\volf{B\left(x_{0},r\right)}\leq BH_{0}^{n}\left(B\left(0,r\right)\right).
\]
Thus
\[
\frac{\volf{B\left(x_{0},2r\right)}}{\volf{B\left(x_{0},r\right)}}\leq\frac{BH_{0}^{n}\left(B\left(0,2r\right)\right)}{AH_{0}^{n}\left(B\left(0,r\right)\right)}\leq\frac{B}{A}\frac{H_{0}^{n}\left(B\left(0,1\right)\right)\cdot\left(2r\right)^{n}}{H_{0}^{n}\left(B\left(0,1\right)\right)\cdot\left(r\right)^{n}}=2^{n}\frac{B}{A}.
\]
Therefore, $\volf{B\left(x_{0},2r\right)}\leq C_D\volf{B\left(x_{0},r\right)}$
with $C_D=2^{n}\underset{K}\sup\,\frac{B}{A}$. Note that $C_D$ is finite by Lemma \ref{laLa} and strictly positive by equation \eqref{eq:lL}.
\end{proof}

\begin{theorem}[Deformation of finite perimeter sets]
 {\label{DL}}
 Let $E\subset M$ be a set of locally finite weighted
perimeter. Assume that $\perf{E,B(p,r)}>0$ for some $p\in M$ and $r>0$. Then there exists a deformation $\left\{ E_{t}\right\} _{t\in\left(-\delta,\delta\right)}$
of $E$, with $E_{0}=E$, by sets of locally finite perimeter, and a constant $C=C(p,r,\delta)$ such that
\begin{enumerate}
\item $\volf{E_t}=\volf{E}+t$.\label{0}
\item $E\triangle E_{t}\subset B\left(p,r\right),$\label{1}
\item $\left|\perf{E}-\perf{E_{t}}\right|\leq C\left|\volf{E}-\volf{E_{t}}
\right|\leq C\left|\volf{E\triangle E_{t}}\right|.$\label{2}
\end{enumerate}
\end{theorem}

\begin{proof} Since $E$ is a set of locally finite perimeter and $\perf{E,B\left(p,r\right)}>0$
there exists  a vector field $X$, with $\left\Vert X\right\Vert \leq1$, such that $X\in\frak{X}_{0}^{\infty}\left(B\left(p,r\right)\right)$ and $\int_E \divv_\Psi(X)\,dM>0$.  Let $\{ \varphi_{s}\}_{s\in\rr}$ be the flow associated to $X$.
Since the map
\[
s\mapsto\volf{\varphi_s(E)}
\]
is differentiable and its derivative at $s=0$ is $\int_E\divv_\Psi(X)dM>0$, we can apply the inverse function theorem to find $\delta>0$ and a function $g:(-\delta,\delta)\to\rr$ such that $g(0)=0$ and $\volf{\varphi_{g(t)}(E)}=\volf{E}+t$.
Let $E_{t}=\varphi_{g(t)}\left(E\right)$. This proves (\ref{0}). If necessary we can reduce $\delta$ so that, for $|t|\leq \de$ we have
\begin{equation}
\label{eq:estvol}
\big|\volf{E_t\cap B(p,r)}-\volf{E\cap B(p,r)}\big|\geq \bigg|\frac{t}{2} \bigg(\underset{E\cap B(p,r)}{\int}\divv_\Psi X\,dM\bigg)^{-1}\bigg|.
\end{equation}

As $\varphi_t(q)=q$ for all $q\not\in B(p,r)$ and $t\in\rr$ we trivially have $E\triangle E_t \subset B(p,r)$ for all $t\in\mathbb{R}$. This proves (\ref{1}).

To prove (\ref{2}) note that, for all $t\in (-\delta,\delta)$, we have
\begin{align*}
\perf{E_t,B(p,r)} & = \perf{\varphi_{g(t)}(E),B(p,r)}=\int_{\ptl_\star\varphi_{g(t)}(E)\cap B(p,r)} e^\Psi  d\h^{n}\\
& = \int_{\ptl_\star E \cap B(p,r)}\frac{(e^\Psi \circ \varphi_t)}{e^\Psi}e^\Psi|\text{Jac}(\varphi_{g(t)})|d\h^{n}
\end{align*}
Hence we have
\begin{align*}
|\perf{E_t,B(p,r)}&-\perf{E,B(p,r)}|=
\\
& = \bigg|\underset{\ptl_\star E \cap B(p,r)}{\int}
\frac{(e^\Psi \circ \varphi_{g(t)})}{e^\Psi}e^\Psi|\text{Jac}(\varphi_{g(t)})|d\h^{n}-
\underset{\ptl_\star E\cap B(p,r)}{\int} e^\Psi  d\h^{n}\bigg|\\
& =  \bigg|\underset{\ptl_\star E \cap B(p,r)}{\int}
\bigg(\frac{(e^\Psi \circ \varphi_{g(t)})}{e^\Psi}|\text{Jac}(\varphi_{g(t)})|-1\bigg)e^\Psi d\h^{n}\bigg|\\
& \leq  \underset {\underset {q\in B(p,r)} {t\in (-\de,\de)}} {\sup}
\bigg|\frac{e^\Psi \circ \varphi_{g(t)} (q)}{e^\Psi (q)}|\text{Jac}(\varphi_{g(t)})(q)|-1\bigg| 
\,\perf{E,B(p,r)}
\end{align*}

Taking into account \eqref{eq:estvol} we have, for $t\in(-\de,\de) \backslash\{0\}$,
\begin{equation*}
\frac{|\perf{E_t,B(p,r)}-\perf{E,B(p,r)}|}{\left|\volf{E_t\cap B(p,r)}-\volf{E\cap B(p,r)}\right|} \leq 
 \frac{2 h(t,p,r,\delta)\cdot\perf{E,B(p,r)}}
 {\left|\bigg( \underset{E\cap B(p,r)}{\int}\divv_\Psi X\,dM \bigg)^{-1}\right|}< C,
\end{equation*}
where
\[
h(t,p,r,\delta)=\frac{1}{|t|}\underset{\underset {t\in(-\delta,\delta)\backslash\{0\}}{q\in B(p,r)}} {\sup}\left|\frac{e^\Psi \circ \vfi_{g(t)}(q)}{e^\Psi(q)}|\Jac(\vfi_{g(t)})(q)|-1\right|
\]
and $C$ is a positive constant which depends on
$p$, $r$ and $\delta$. note that (\ref{2}) is trivially true for $t=0$ and any constant $C$. 
So we have
\[
|\perf{E_t,B(p,r)}-\perf{E,B(p,r)}|\leq C \left|\volf{E_t\cap B(p,r)}-\volf{E\cap B(p,r)}\right|,
\]
and this fact together with (\ref{1}) implies the inequality of the left side of (\ref{2}). Note that (\ref{2}) is trivially true for $t=0$ and any constant $C$.

On the other hand, for any positive measure $\mu$ we have
\[
\big|\mu(E)-\mu(E')\big|\le \mu(E\triangle E').
\]
This completes the proof of (\ref{2}).
\end{proof}

To conclude this section, we sketch the proof of Theorem \ref{thm:existence}. Since this proof has become standard after \cite{gr}, we include some basic guidelines for reader's convenience. 

\begin{proof}[Proof of Theorem~\ref{thm:existence}] Using the relative isoperimetric inequality \eqref{eq:R}, the doubling property \eqref{eq:DP} and the hypothesis that $\Isom(M,g,\Psi)$ acts cocompactly on $M$, an isoperimetric inequality for small volumes can be obtained in a standard way, see Lemma~3.10 in \cite{gr}. Combining the latter with the deformation property of finite perimeter sets proven in Theorem~\ref{DL}, and using again that $\Isom(M,g,\Psi)$ acts cocompactly, we can prove that the isoperimetric solutions are bounded, see Lemma 4.6 in \cite{gr}. The Structure Theorem for minimizing sequences of sets of positive volume $v>0$, see Proposition 5.1 in \cite{gr}, works also in our case without modification. From them the Concentration Lemma~6.2 in \cite{gr}, and the Existence Theorem~6.1 in \cite{gr} work without relevant modifications.
\end{proof}

\section{Isoperimetric regions in $M\times\rr$}
\label{sec:slabs}

In this section we prove existence of isoperimetric regions in a cylinder with density for large volumes.
We shall need the following preliminary results in the proof of Theorem~\ref{thm:cylinder}.

\begin{lemma}
{\label{DIVG}}
Let $(M,g,\Psi)$ be a compact manifold with density. Then there exist constants $c_1,c_2>0$, only depending on $M$, such that, for any set $E\subset M$ of finite perimeter with $\volpsi(E)\le\volpsi(M)/2$ we have
\begin{enumerate}
\item $\perf{E}\ge c_1\volpsi(E)$, and
\item $\perf{E}\ge c_2\volpsi(E)^{(n-1)/n}$.
\end{enumerate}
\end{lemma}

\begin{proof}
It follows easily since the isoperimetric profile of $(M,g,\Psi)$ is strictly positive and asymptotic to the function $t\mapsto t^{(n-1)/n}$ for $t>0$ small.
\end{proof}

We say that $E\subset N=M\times\rr$ is a \emph{normalized set} if the intersection $E_p=E\cap (\{p\}\times\rr)$ is either empty or a vertical segment centered at $(p,0)$ for all $p\in M$. Notice that a normalized set is invariant by the reflection $\sigma:M\times\rr \to M \times \rr$ defined by $\sigma(p,t)=(p,-t)$, an isometry of $(M\times\rr,g\times g_0)$ preserving the weighted volume. Given any set $E\subset N$, we denote by $E^*$ the projection of $E$ over $M$. From now on we denote $(E_t)^*$ by $E^*_t$ to simplify the notation. Notice that for normalized sets one has $E_{t}^*\subset E_s^*$ whenever $|s|\le |t|$. We denote by $P$ and $\volume$ the perimeter and volume in the manifold with density $(M\times\rr,g\times g_0,\Psi\times 1)$.

\begin{lemma}
\label{lem:linear}
If $E\subset N$ is a normalized isoperimetric region and $\volpsi(M\setminus E^*)>0$ then there exists a constant $c>0$ independent of $\volume(E)$ such that
\begin{equation}
\label{eq:linearisop}
P(E)\ge c\volume(E).
\end{equation}
\end{lemma}

\begin{proof}
For every $t\in\rr$ we define $M_t=M\times\{t\}$ and $E_t=E\cap M_t$. As $E$ is normalized we can choose $\tau\ge 0$ so that $\volpsi(E_t^*)\le \volpsi(M)/2$ for all $t\ge \tau$ and $\volpsi(E_t^*)> \volpsi(M)/2$ for all $t\in [0,\tau)$ if $\tau>0$.

Let us consider first the case $\tau>0$.

We apply the coarea formula and Lemma \ref{DIVG} (1) to obtain
\begin{equation}
\label{eq:eq1}
\begin{split}
P(E)\ge P(E,M\times [\tau,\infty))&\ge \int_\tau^\infty \perf{E_s^*}ds
\\
&\ge c_1\int_\tau^\infty\volpsi(E_s^*)ds
\\
&=c_1\volume(E\cap (M\times [\tau,\infty))).
\end{split}
\end{equation}

On the other hand, for $t\in [0,\tau)$ we have
\begin{equation}
\label{eq:4-2}
\volpsi(M\setminus E_t^*)\ge P(E,M\times (0,t)),
\end{equation}
since otherwise
\begin{align*}
\volpsi(M)&=\volpsi(M\setminus E_t^*)+\volpsi(E_t^*)
\\
&<P(E,M\times (0,t))+P(E,M\times (t,\infty))
\\
&\le P(E)/2.
\end{align*}
This is a contradiction since comparison of $E$ with a slab $M\times [a,b]$ of the same volume implies that $P(E)\le P(M\times [a,b])=2\volpsi(M)$. This proves \eqref{eq:4-2}. Calling $y(t)=\volpsi(M\setminus E_t^*)$, using the coarea formula and Lemma~\ref{DIVG}$(2)$, we may rewrite the inequality \eqref{eq:4-2} as
\[
y(t)\ge c_2\int_0^t y(s)^{(n-1)/n}ds.
\]
As $y(t)>0$ for all $t\in [0,\tau)$ we have 
\[
y(t)\ge\big(\tfrac{c_2}{n}\big)^nt^n.
\]
%
In particular, taking limits when $t\to\tau^-$ and using that $y(t)$ is non-decreasing
\[
\volpsi(M)\ge \volpsi(M\setminus E_\tau^*)=y(\tau)\ge \big(\tfrac{c_2}{n}\big)^n\tau^n.
\]
Hence
\[
\tau\le\frac{n\volpsi(M)^{1/n}}{c_2}
\]
and so
\begin{equation}
\label{eq:eq2}
\begin{split}
\volpsi(E\cap (M\times (0,\tau))&=\int_0^\tau\volpsi(E_s^*)ds\le\volpsi(E_0^*)\tau
\\
&\le\volpsi(E_0^*)\frac{n\volpsi(M)^{1/n}}{c_2}
\\
&\le \frac{n\volpsi(M)^{1/n}}{c_2}\frac{P(E)}{2}.
\end{split}
\end{equation}
The last inequality follows $2\volpsi(E_0^*)\le P(E)$, which holds since $E$ is normalized and so $P(E)$ is the sum of a lateral area that projects to some set of weighted measure zero on $M$ and the area of the graphs of two  $C^1$ functions $u$ and $-u$ over some set $\Om\subset M$ of full measure in $E^*$. So we have
\[
\volpsi(E^*)=\volpsi(\Om)=\int_\Om e^\Psi dM\le \int_\Om e^\Psi\sqrt{1+|\nabla u|^2}dM\le \frac{P(E)}{2}.
\]
Hence \eqref{eq:linearisop} follows from \eqref{eq:eq1} and \eqref{eq:eq2}.

It remains to consider the case $\tau=0$. In this case, equation \eqref{eq:eq1} alone implies the linear isoperimetric inequality \eqref{eq:linearisop} since $\volf{E\cap(M\times[0,+\infty))}=\frac{1}{2}\volf{E}$ as $E$ is normalized.
\end{proof}

\begin{theorem}
\label{thm:cylinder}
Let $(M,g,\Psi)$ be a compact manifold with density. For large volumes, isoperimetric regions in the cylinder $(M\times\rr,g\times g_0,\Psi\times 1)$
are slabs of the form $M\times [a,b]$, where $[a,b]\subset\rr$ is a bounded interval. 
\end{theorem}

\begin{proof}
Existence of isoperimetric regions in $N=M\times\rr$ is guaranteed by Theorem~\ref{thm:existence}. If $E$ is an isoperimetric region in $N$, comparison with slabs implies
\begin{equation}
\label{eq:compslabs}
P(E)\le 2\volpsi(M),
\end{equation}
for all volumes $v>0$. 

We take an isoperimetric set $E\subset M$. Let $\symm(E)$ be its Steiner symmetrization with respect to $M\times\{0\}$, see \cite[\S~14.1]{MR2976521}. As
\[
\volume(E)=\int_M\bigg\{\int_{E_p} e^{\Psi\times 1}dt\bigg\}dM=\int_Me^{\Psi(p)} |E_p|dM(p)=\volume(\symm(E)),
\]
where $|E_p|$ is the $1$-dimensional Lebesgue measure of $E_p$, the volume is preserved when we pass to the Steiner symmetrization of $E$. To see that
\begin{equation}
\label{eq:persym}
P(\symm(E))\le P(E)
\end{equation}
we consider a function $u:\Omega\subset M\to\rr$ and the graph $G(u)$ of $u$ and we observe that the weighted area of $G(u)$ is given by
\[
\text{area}_\Psi(G(u))=\int_\Omega e^\Psi\sqrt{1+|\nabla u|^2}dM.
\]
So we can reason as in the proof of the Euclidean case to verify \eqref{eq:persym}, see again \cite[\S~14.1]{MR2976521}. Equality holds if and only if $E=\symm(E)$. If $E$ is an isoperimetric region then also $\symm(E)$ is isoperimetric and, moreover, $\symm(E)^*=E^*$. So from now on we assume that $E$ is normalized replacing $E$ by $\symm(E)$ if necessary.

If $\volpsi(M\setminus E^*)>0$ then Lemma~\ref{lem:linear} provides a constant $c>0$ independent of $\volume(E)$ so that $P(E)\ge c\volume(E)$. But this in contradiction to \eqref{eq:compslabs}, since by hypothesis we are working with large volumes. Hence $\volpsi(M\setminus E^*)=0$ and $E^*=M$ and $E$ is the region between the graphs of two functions $u,v:M\to\rr$. By regularity of isoperimetric regions, $\nabla u$, $\nabla v$ are defined a.e. on $M$ and
\begin{align*}
P(E)&=\int_M e^\Psi \sqrt{1+|\nabla u|^2}dM+\int_M e^\Psi \sqrt{1+|\nabla v|^2}dM
\\ 
&\ge 2\int_Me^\Psi dM=2\volpsi(M).
\end{align*}
Since $P(E)\le 2\volpsi(M)$ we should have equality in the above inequality, that implies $\nabla u=\nabla v=0$ and so $E$ is a slab. This completes the proof of the Theorem.
\end{proof}

\begin{remark}
Note that Theorem \ref{thm:cylinder} does not hold when there is a non-trivial density in the vertical factor. For example, in $(S^n \times \rr,g\times g_0, 1\times e^{-t^2/2})$ the only isoperimetric regions are of the type $S^n \times(-\infty,a)$ or $S^n \times (a,\infty)$. See Example 4.6 in \cite{ROSALES2021}.
\end{remark}

%


\end{document}